\documentclass[12points,reqno]{amsart}
\usepackage{amsfonts}
\usepackage{amssymb}
\usepackage{graphicx}
\usepackage{amsmath}

\newtheorem{theorem}{Theorem}
\newtheorem{lemma}{Lemma}

\theoremstyle{definition}
\newtheorem{definition}{Definition}
\newtheorem{example}{Example}

\theoremstyle{remark}
\newtheorem{remark}[theorem]{Remark}

\numberwithin{equation}{section}

\everymath{\displaystyle}

\begin{document}
\title{ Finite generation of Lie algebras associated to associative algebras}

\address{
a. Department of Mathematics,
King Abdulaziz University,
Jeddah, SA\\
\newline\email{analahmadi@kau.edu.sa; hhaalsalmi@kau.edu.sa}
\newline
b. Department of Mathematics\\
Ohio University, Athens, USA \\
\newline\email{jain@ohio.edu}
\newline
c. Department of Mathematics\\
University of California, San Diego, USA\\
\newline\email{ezelmano@math.ucsd.edu}
\newline
1. To whom correspondence should be addressed\\ E-mail: ezelmano@math.ucsd.edu.
\newline
Author Contributions: A. A., H. A., S. K. J., E. Z. designed research;
performed research and wrote the paper. The authors declare no conflict of interest. }

\author{Adel Alahmedi $^a$, Hamed Alsulami$^a$, S. K. Jain $^{a,b}$
, Efim Zelmanov$^{a,c,1}$}

%
%


\maketitle

\begin{abstract}
 Let $F$ be a field of characteristic not $2$ . An associative $F$-algebra $R$ gives rise to the commutator Lie algebra $R^{(-)}=(R,[a,b]=ab-ba).$ If the algebra $R$ is equipped with an involution $*:R\rightarrow R$ then the space of the skew-symmetric elements $K=\{a \in R \mid a^{*}=-a \}$ is a Lie subalgebra of $R^{(-)}.$ In this paper we find sufficient conditions for the Lie algebras $[R,R]$ and $[K,K]$ to be finitely generated.
\end{abstract}

\maketitle

\section{Introduction}
Let $F$ be a field of characteristic not $2.$ An associative $F$-algebra $R$ gives rise to the commutator Lie algebra $R^{(-)}=(R,[a,b]=ab-ba)$ and the Jordan algebra $R^{(+)}=(R,a\circ b=\frac{1}{2}(ab+ba)).$ If the algebra $R$ is equipped with an involution $*:R\rightarrow R$ then the space of skew-symmetric elements $K=\{ a \in R \mid a^{*}=-a\}$ is a Lie subalgebra of $R^{(-)},$ the space of symmetric elements $H=\{ a \in R \mid a^{*}=a\}$ is a Jordan subalgebra of $R^{(+)}.$ Following the result of J.M. Osborn (see[4]) on finite generation of the Jordan algebras $R^{(+)},H,$  I. Herstein [4] raised the question about finite generation of Lie algebras associated to $R.$ In this paper we find sufficient conditions for the Lie algebras $[R^{(-)},R^{(-)}],[K,K]$ to be finitely generated.
\begin{theorem}\label{th1}
Let $R$ be a finitely generated associative F-algebra with an idempotent $e$
such that $ReR=R(1-e)R=R$. Then the Lie algebra $[R,R]$ is finitely generated.
\end{theorem}
The following example shows that the idempotent condition can not be dropped.
\begin{example}\label{ex1}
The algebra $R=%
\begin{pmatrix}
F[x] & F[x] \\
0 & F[x]%
\end{pmatrix}%
$ of triangular $2\times 2$ matrices over the polynomial algebra $F[x]$ is
finitely generated. However the Lie algebra

$[R, R]=%
\begin{pmatrix}
0 & F[x] \\
0 & 0%
\end{pmatrix}%
$ is not.
\end{example}
\begin{theorem}\label{th2}
Let $R$ be a finitely generated associative F-algebra with an involution $\ast :R\rightarrow R$. Suppose that $R$ contains an idempotent $e$ such that
$ee^{\ast }=e^{\ast }e=0$ and $ReR=R(1-e-e^{\ast })R=R$. Then the Lie algebra $[K,K]$ is finitely generated.
\end{theorem}
The following example shows that the condition on the idempotent cannot be
relaxed.

\begin{example}\label{ex2}
Consider the associative commutative algebra $A=F[x$, $y]/id(x^{2})$ with
the automorphism $\varphi $ of order $2$: $\varphi (x)=-x$, $\varphi (y)=y$.
The algebra $R=M_{2}(A)$ of $2\times 2$ matrices over $A$ has an involution $
\begin{pmatrix}
a & b \\
c & d%
\end{pmatrix}%
\rightarrow
\begin{pmatrix}
d^{\varphi } & b^{\varphi } \\
c^{\varphi } & a^{\varphi }%
\end{pmatrix}%
$. We have $[K, K]\leq xM_{2}(F[y])$, $dim_{F}[K,K]=\infty $, which
implies that algebra $[K, K]$ is not finitely generated.
\end{example}

W.E.Baxter [2] showed that if $R$ is a simple  $F-$ algebra, which is not $\leq 16$ dimensional over its center $Z$ then the Lie algebra $[K,K]/[K,K]\bigcap Z$ is simple.

\begin{theorem}\label{th3}
Let $R$ be a simple finitely generated $F-$ algebra with an involution $*:R\rightarrow R.$ Suppose that $R$ contains an idempotent $e$ such that $ee^{*}=e^{*}e=0.$ Then the Lie algebra $[K,K]/[K,K]\bigcap Z$ is finitely generated.
\end{theorem}
\bigskip

\section{Finite generation of Lie algebras $[R,R]$}

Consider the Peirce decomposition $R=eRe+eR(1-e)+(1-e)Re+(1-e)R(1-e)$. The
components $eR(1-e),(1-e)Re$ lie in $[R$, $R]$ since $eR(1-e)=[e,eR(1-e)]$, $%
(1-e)Re=[e,(1-e)Re]$.

\begin{lemma}\label{lem1}
The Lie algebra $[R$, $R]$ is generated by $eR(1-e)+(1-e)Re$.
\end{lemma}

\begin{proof}
We only need to show that $$[eRe,eRe]+[(1-e)R(1-e),(1-e)R(1-e)]\subseteq  \text { Lie } \left\langle eR(1-e),(1-e)Re\right\rangle .$$
From $R=R(1-e)R$ it follows that $eRe=eR(1-e)Re$. Hence an arbitrary element
from $eRe$ can be represented as a sum $\sum\limits_{i} a_{i}b_{i},  a_{i}\in eR(1-e), b_{i}\in (1-e)Re$.
Now $a_{i}b_{i}=a_{i}\circ b_{i}+\frac{1}{2}[a_{i}, b_{i}]$, where $x\circ y=\frac{1}{2}(xy+yx)$.
For an arbitrary element $c\in eRe$ we have
$[a_{i}\circ b_{i}, c]=[a_{i}, b_{i}\circ c]+[b_{i}, a_{i}\circ c]\in[eR(1-e), (1-e)Re]$ and
$[[a_{i}, b_{i}], c]=[a_{i}, [b_{i}, c]]-[b_{i}, [a_{i}, c]]\in[ eR(1-e), (1-e)Re]$.\vspace{.5cm}
We showed that $[eRe,eRe]\subseteq [ eR(1-e), (1-e)Re]$.
The inclusion $[(1-e)R(1-e), (1-e)R(1-e)]\subseteq [ eR(1-e),(1-e)Re] $ is proved similarly. Lemma is proved.
\end{proof}

\begin{definition}
A pair of vector spaces $(A^{-}, A^{+})$ with trilinear products
\newline $A^{+}\times A^{-}\times A^{+}\rightarrow A^{+}, A^{-}\times A^{+}\times A^{-}\rightarrow A^{-},$
$a^{\sigma }\times b^{-\sigma }\times c^{\sigma}\mapsto (a^{\sigma }, b^{-\sigma }, c^{\sigma })\in A^{\sigma },$
$\sigma =+$ or $-$, is called an \textit{associative pair }if it satisfies the identities
$$((x^{\sigma }, y^{-\sigma }, z^{\sigma }), u^{-\sigma }, v^{\sigma })=(x^{\sigma }, (y^{-\sigma }, z^{\sigma }, u^{-\sigma }),
v^{\sigma })=(x^{\sigma }, y^{-\sigma }, (z^{\sigma }, u^{-\sigma },v^{\sigma }))$$
\end{definition}

\begin{example}\label{ex4}
The pair of Peirce components $(eR(1-e), (1-e)Re)$ is an associative pair
with respect to the operations $(a^{\sigma }, b^{-\sigma }, c^{\sigma})=a^{\sigma }b^{-\sigma }c^{\sigma }$.
\end{example}

\begin{lemma}\label{lem2}
 Let $R$ be a finitely generated algebra and let $e$, $f\in R$ be
idempotents such that $ReR=RfR=R$. Then the associative pair $P=(eRf, fRe)$
is finitely generated.
\end{lemma}

\begin{remark}
In [8] it is proved that if $R$ is a finitely generated algebra, $e\in R$ is
an idempotent such that $ReR=R$ then the Peirce component $eRe$ is a
finitely generated algebra.
\end{remark}

\begin{proof}
Suppose that the algebra $R$ is generated by elements $a_{1},\ldots , a_{m}$. Suppose further that
$a_{i}=\sum\limits_{k}\alpha_{ik}u_{ik}ev_{ik}=\sum\limits_{t}\beta _{it}u_{it}^{\prime }fv_{it}^{\prime },$
where $1\leq i\leq m;$ $\alpha _{ik},\beta _{it}\in F;$ $u_{ik}, v_{ik},u_{it}^{\prime }, v_{it}^{\prime }$
are products in generators $a_{1},\ldots,a_{m}.$ Let $d$ denote the maximum lengths of the products
$u_{ik},v_{ik}, u_{it}^{\prime }, v_{it}^{\prime }$ for all $i, k, t$.
We claim that the pair $P$ is generated by elements $euf$, $fue$, where $u$
runs over all products in $a_{1},\ldots, a_{m}$ of length $\leq 3d+1$.
To prove the claim we need to show that for an arbitrary product
$u=a_{i_{1}}\cdots a_{i_{N}}$ of length $N>3d+1$ the elements $euf$, $fue$
lie in the subpair generated by $evf$, $fve$, where $v$ runs over all
products in $a_{1},\ldots, a_{m}$ of length $<N$.
\vspace{.5cm}
There exist integers $N_{1}, N_{2}, N_{3}$ such that $N/3-1<N_{i}\leq N/3 ,1\leq i\leq 3$ and $N=N_{1}+N_{2}+N_{3}+2$.
Let $u=u_{1}a_{i}u_{2}a_{j}u_{3}$, length $(u_{i})=N_{i}$, $1\leq i\leq 3$.
Then
$a_{i}=\sum\limits_{k}\beta _{k}p_{k}^{\prime }fq_{k}^{\prime }, a_{j}=\sum\limits_t\gamma _{t}p_{t}^{\prime \prime }eq_{t}^{\prime\prime }, $
where $\beta _{k},\gamma _{t}\in F;$ $p_{k}^{\prime }, q_{k}^{\prime }, p_{t}^{\prime \prime }, q_{t}^{\prime \prime }$ are
products in $a_{1}, \ldots , a_{m}$ of length $\leq d$.
Now $euf=\sum\limits_{k,t}\beta _{k}\gamma_{t}eu_{1}p_{k}^{\prime }fq_{k}^{\prime }u_{2}p_{t}^{\prime \prime}eq_{t}^{\prime \prime }u_{3}f.$
The lengths of the products $u_{1}p_{k}^{\prime }, q_{k}^{\prime }u_{2}p_{t}^{\prime \prime },q_{t}^{\prime \prime }u_{3}$  are less than $N$. The element $fue$ is treated similarly.
Lemma is proved.
\end{proof}

\begin{remark}
We don't assume that the algebra $R$ of Theorem 1 is unital. However passing
to the unital hull we see that if $R$ is a finitely generated algebra, $e\in R$ is an idempotent such that $ReR=R(1-e)R=R$ then the associative pair
$(eR(1-e),(1-e)Re)$ is finitely generated.
\end{remark}
We will need some definitions from Jordan theory.

\begin{definition}
An algebra over a field $F$ of characteristic $\neq 2$ with
multiplication $a\circ b$ is called a \textit{Jordan algebra} if it
satisfies the identities
\begin{itemize}
  \item[(J1)] $x\circ y=y\circ x$
  \item[(J2)] $(x^{2}\circ y)\circ x=x^{2}\circ (y\circ x).$
\end{itemize}
\end{definition}

For references on Jordan algebras see [5, 7, 9].

An arbitrary associative algebra $A$ gives rise to the Jordan algebra $$A^{(+)}=(A, a\circ b=\frac{1}{2}(ab+ba)).$$

\begin{definition}
A pair of F-spaces $P=(P^{-}$, $P^{+})$ with trilinear products
\newline $P^{\sigma}\times P^{-\sigma }\times P^{\sigma }\rightarrow P^{\sigma }$,
$a^{\sigma}\times b^{-\sigma }c^{\sigma }\rightarrow \{a^{\sigma }, b^{-\sigma }, c^{\sigma }\}\in P^{\sigma };$ $\sigma =+$ or $-;$
\newline $a^{\sigma }, c^{\sigma }\in P^{\sigma }, b^{-\sigma }\in P^{-\sigma },$
is called a Jordan pair if it satisfies the following identities and all their linearizations:
\begin{itemize}
  \item[(J1)] $\{x^{\sigma }, y^{-\sigma },\{x^{\sigma }, z^{-\sigma }, x^{\sigma }\}\}=\{x^{\sigma }, \{y^{-\sigma }, x^{\sigma }, z^{-\sigma}\},x^{\sigma }\}$,
  \item[(J2)] $\{\{x^{\sigma },y^{-\sigma },x^{\sigma }\},y^{-\sigma },z^{-\sigma }\}=\{x^{\sigma },\{y^{-\sigma },x^{\sigma },y^{-\sigma}\},z^{\sigma}\}$,
  \item[(J3)] $\{\{x^{\sigma},y^{-\sigma },x^{\sigma }\},z^{-\sigma },\{x^{\sigma },y^{-\sigma },x^{\sigma }\}\}=\{x^{\sigma },\{y^{-\sigma },\{x^{\sigma},z^{-\sigma }, x^{\sigma }\},y^{-\sigma }\},x^{\sigma }\}$.
\end{itemize}

An arbitrary associative pair $A=(A^{-}, A^{+}), a^{\sigma }\times b^{-\sigma }\times c^{\sigma }\rightarrow (a^{\sigma }, b^{-\sigma }, c^{\sigma
})\in A^{\sigma }$ gives rise to the Jordan pair $A^{(+)}=(A^{-}, A^{+})$
with operations $\{a^{\sigma }, b^{-\sigma }, c^{\sigma }\}=(a^{\sigma },b^{-\sigma }, c^{\sigma })+(c^{\sigma }, b^{-\sigma }, a^{\sigma
})\in A^{\sigma }$, $\sigma =+$ or $-$.
\end{definition}

For further properties of Jordan pairs see [6].

J.M. Osborn (see [4]) showed that a finitely generated associative algebra $R$
gives rise to the finitely generated Jordan algebra $R^{(+)}$.

\begin{lemma}
Let $A=(A^{-},A^{+})$ be a finitely generated associative pair. Then the Jordan pair $A^{(+)}$ is also finitely generated.
\end{lemma}

\begin{proof}
Suppose that the pair $A$ is generated by elements $a_{1}^{+}, \ldots , a_{m}^{+}\in A^{+};$ \newline $a_{1}^{-}, \ldots ,a_{m}^{-}\in A^{-}$. Consider an
(associative) product $a=a_{i_{1}}^{+}a_{j_{1}}^{-}a_{i_{2}}^{+}\cdots a_{j_{s}}^{-}a_{i_{s+1}}^{+}\in A^{+}.$
We claim that if the indices $i_{1}, \ldots , i_{s+1}$ are not all distinct then $a$ is a Jordan expression in shorter products.
We have
\begin{alignat*}{5}
x^{+}y^{-}u^{+}v^{-}u^{+}&=(x^{+}y^{-}u^{+}v^{-}u^{+}+u^{+}v^{-}x^{+}y^{-}u^{+})-u^{+}v^{-}x^{+}y^{-}u^{+}\\
                          &=\{x^{+}y^{-}u^{+}, v^{-}u^{+}\}-\frac{1}{2}\{u^{+}, v^{-}x^{+}y^{-}u^{+}\}\, \tag{1}
\end{alignat*}
Linearizing this equality in $u^{+}$ (see [9]) we get
\begin{equation*}
x^{+}y^{-}\{u_{1}^{-}, v^{-}, u_{2}^{+}\}=\{x^{+}y^{-}u_{1}^{+}v^{-} , u_{2}^{+}\}+\{x^{+}y^{-}u_{2}^{+}, v^{-}, u_{1}^{+}\}-\{u_{1}^{+}, v^{-}x^{+}y^{-}, u_{2}^{+}\}\,\tag{2}.
\end{equation*}
Now (1) and (2) imply
\begin{alignat*}{5}
x^{+}y^{-}u^{+}v^{-}u^{+}z^{-}t^{+}&=x^{+}y^{-}(u^{+}v^{-}u^{+}z^{-}t^{+}+t^{+}z^{-}u+v^{-}u^{+})-x^{+}y^{-}t^{+}z^{-}u^{+}v^{-}u^{+}\\
&=\frac{1}{2}x^{+}y^{-}\{\{u^{+}v^{-}u^{+}\}, z^{-}, t^{+}\}-x^{+}y^{-}t^{+}z^{-}u^{+}v^{-}u^{+}\\
&=\frac{1}{2}x^{+}y^{-}\{\{u^{+}v^{-}u^{+}\}, z^{-}, t^{+}\}+\frac{1}{2}\{x^{+}y^{-}t^{+}, z^{-}, \{u^{+}, v^{-}, u^{+}\}\}\\
&-\frac{1}{2}\{\{u^{+}, v^{-},u^{+}\}, z^{-}x^{+}y^{-}, t^{+}\}-\{x^{+}y^{-}t^{+}z^{-}u^{+}, v^{-}, u^{+}\}\\
&+\frac{1}{2}\{u^{+},z^{-}t^{+}y^{-}x^{+}v^{-}, u^{+}\}\,\tag{3}
\end{alignat*}
If $i_{p}=i_{q}$, $p<q$, then $a$ is an expression of the same type as the
left hand sides of (1), (3). Hence $a$ is a Jordan expression in shorter products.

We showed that the Jordan pair $A^{(+)}$ is generated by elements $%
a_{i_{1}}^{+}a_{j_{1}}^{-}a_{i_{2}}^{+}\cdots a_{i_{s+1}}^{+}$, where the
indices $i_{1}, \ldots , i_{s+1}$ are distinct, and elements $%
a_{j_{1}}^{-}a_{i_{1}}^{+}\cdots a_{i_{s}}^{+}a_{j_{s+1}}$, where $j_{i}
,\ldots , j_{s+1}$ are distinct. Hence the pair $A^{(+)}$ is finitely
generated.
\end{proof}

\subsubsection*{Proof of Theorem 1.}
\begin{proof}
Let $R$ be a finitely generated associative algebra with an idempotent $e$ such that $R=ReR=R(1-e)R$. By Lemma 2 the associative
pair \newline $A=(eR(1-e), (1-e)Re)$ is finitely generated. By Lemma 3 the Jordan
pair $A^{(+)}$ is finitely generated as well. For arbitrary elements $a^{\sigma },c^{\sigma }\in A^{\sigma },b^{-\sigma }\in A^{\sigma }$,
$\sigma =+$ or $-$, we have $\{a^{\sigma },b^{-\sigma }, c^{\sigma}\}=[[a^{\sigma },b^{-\sigma }],c^{\sigma }]$. This implies that
generators of the Jordan pair $A^{(+)}$ generate
the Lie algebra $A^{-}+[A^{-},A^{+}]+A^{+}$. Now it remains to recall that
$A^{-}+[A^{-}, A^{+}]+A^{+}=[R,R]$ by Lemma 1. Theorem is proved.
\end{proof}

\bigskip

\section{Finite generation of Lie algebras $[K,K].$}

Let $R$, $\ast :R\rightarrow R$, be an involutive algebra with an idempotent
$e$ satisfying $ee^{\ast }=e^{\ast }e=0$ and $ReR=R(1-e-e^{\ast })R=R$. Let $s=1-e-e^{\ast }\neq 0$.
For an arbitrary element $a\in R$ we denote $\{a\}=a-a^{\ast }\in K$.
Let $R_{-2}=eRe^{\ast },$ $R_{-1}=eRs+sRe^{\ast }$, $R_{0}=eRe+e^{\ast
}Re^{\ast }+sRs$, $R_{1}=e^{\ast }Rs+sRe$, $R_{2}=e^{\ast }Re$. Then $R=R_{-2}+R_{-1}+R_{0}+R_{1}+R_{2}$ is a $\mathbb{Z}$-grading.

Denote $K_{i}=K\cap R_{i},$ $H_{i}=H\cap R_{i},$ where $H=\{a\in R|a^{\ast }=a\}.$

\begin{remark}
By Lemma 2 the associative pair $(R_{-2}, R_{2})$ is finitely generated.
The restriction of $\ast $ is an involution of the pair $(R_{-2}, R_{2})$.
However, $-\ast $ is also an involution of $(R_{-2}, R_{2})$. The Jordan pair $(K_{-2}, K_{2})$ is
$K((R_{-2},R_{2}),\ast )=H((R_{-2}, R_{2}), -\ast )$. An analog of
Lemma 3 for associative pairs is not true: the Jordan pair of symmetric
elements of a finitely generated involutive associate pair may be not
finitely generated. An example can be derived from the Example 2 above.
\end{remark}
\bigskip

\begin{lemma}
$K_{2}=[K_{1}, K_{1}],$ $H_{2}=\text{span}_{F}\{k^{2} | k\in K_{1}\}$ and,
similarly, $K_{-2}=[K_{-1}, K_{-1}],$ $H_{-2}=\text{span}_{F}\{k^{2} | k\in K_{-1}\}$.
\end{lemma}

\begin{proof}
Recall that the algebra $R$ is generated by the elements $a_{1},\ldots,a_{m}$.
An arbitrary generator $a_{i}$ can be represented as $a_{i}=\sum\limits_{j}\alpha _{ij}v_{ij}sw_{ij}$, where $\alpha _{ij}\in F;$ $v_{ij},$ $
w_{ij}$ are products in generators $a_{1},\ldots ,a_{m}$ (may be, empty).
Let $a=a_{i_{1}}\cdots a_{i_{r}}$ be an arbitrary product of generators.
Applying the equalities above to the generator $a_{i_{1}}$ we get
\begin{alignat*}{5}
e^{\ast }ae&=\sum \alpha _{i_{1}j}e^{\ast}v_{i_{1}j}sw_{i_{1}j}a_{i_{2}}\cdots a_{i_{r}}e\\
&=\sum \alpha _{i_{1}j}\{e^{\ast }v_{i_{1}j}s\}\{sw_{i_{1}j}a_{i_{2}}\cdots a_{i_{r}}e\}\in K_{1}K_{1}.
\end{alignat*}

Now
$K_{2}=\{e^{\ast }Re\}=\{K_{1}K_{1}\}=[K_{1},K_{1}];$
\begin{alignat*}{5}
H_{2}&=\{e^{\ast }ae+e^{\ast }a^{\ast }e |a\in R\}\\
&=\{k_{1}k_{2}+(k_{1}k_{2})^{\ast }|k_{1},k_{2}\in K_{1}\}\\
&=\text{ span} _{F}\{k^{2} |k\in K_{1}\}.
\end{alignat*}
 Lemma is proved.
\end{proof}

\bigskip

\begin{lemma}
There exists a finite subset $M_{-1}\subset K_{-1}$ such that $
R_{1}=M_{-1}R_{2}+R_{2}M_{-1}$. Similarly, there exists a finite subset $
M_{1}\subset K_{1}$ such that $R_{-1}=M_{1}R_{-2}+R_{-2}M_{1}$.
\end{lemma}

\begin{proof}
Represent each generator $a_{i}$ as $a_{i}=\sum\limits_{j}\alpha_{ij}^{\prime }v_{ij}^{\prime }ew_{ij}^{\prime },$ where
$\alpha_{ij}^{\prime }\in F;$ $v_{ij}^{\prime },$ $w_{ij}^{\prime }$ are products in $a_{1},\ldots , a_{m}$(may be, empty).
Consider an element $e^{\ast}as $, where $a=a_{i_{1}}\cdots a_{i_{r}}$ and apply the decomposition above to $a_{i_{r}}$. We'll get
\begin{alignat*}{5}
e^{\ast }as&=\sum \alpha _{i_{r}j}^{\prime }e^{\ast }a_{i_{1}}\cdots a_{i_{r-1}}v_{i_{r}j}^{\prime }ew_{i_{r}j}^{\prime }s\\
&=\sum \alpha_{i_{r}j}^{\prime }e^{\ast }a_{i_{1}}\cdots a_{i_{r-1}}v_{i_{r}j}^{\prime}\{ew_{i_{r}j}^{\prime }s\}.
\end{alignat*}
It remains to choose $M_{-1}=\{\{ew_{i_{r}j}^{\prime }s\}\}\subseteq K_{-1}.$ Lemma is proved.
\end{proof}

\bigskip

\begin{lemma}
\begin{itemize}
  \item[(1)]$\sum\limits_{i\neq 0}K_{i}\subseteq [K,K];$
  \item[(2)] $[K,K]$ is generated by $K_{-1},K_{1}$.
\end{itemize}
\end{lemma}

\begin{proof}
\begin{itemize}
\item[(1)] For arbitrary elements $k_{1}=\{sae\}, k_{-1}=\{sae^{\ast }\},$ $a\in R,$
we have $k_{1}=[k_{1}, e-e^{\ast }],$ $k_{-1}=[e-e^{\ast },k_{-1}]$.
Hence $K_{-1},K_{1}\subseteq [K,K]$.
By Lemma 4, $K_{2}=[K_{1},K_{1}],$ $K_{-2}=[K_{-1},K_{-1}]\subseteq[K, K].$
\item[(2)] We have $eRe=eRe^{\ast}Re=R_{-2}R_{2}=K_{-2}K_{2}+K_{-2}H_{2}+H_{-2}K_{2}+H_{-2}H_{2}$.
Hence $\{eRe\}\subseteq [K_{-2},K_{2}]+[H_{-2},H_{2}]+K_{-2}\circ H_{2}+H_{-2}\circ K_{2}$.
This implies that $[\{eRe\},K_{0}]\subseteq [K_{-2},K_{2}]+[H_{-2},H_{2}]\subseteq[K_{-2},K_{2}]+[K_{-1},K_{1}]$ by Lemma 4.
Now $sRs$ is spanned by elements of the type $saebs;$ $a$, $b\in R$. We have
$saebs=(sae-e^{\ast }a^{\ast }s)(ebs-sb^{\ast }e^{\ast })-e^{\ast }a^{\ast}sb^{\ast }e^{\ast }\in K_{1}K_{-1}+e^{\ast }Re^{\ast }$.
Hence $\{sRs\}\subseteq [K_{-1},K_{1}]+\{eRe\}$. Now $[\{sRs\},K_{0}]\subseteq [K_{-1},K_{1}]+[K_{-2},K_{2}]$ by what we
proved above. Lemma is proved.
\end{itemize}
\end{proof}


\begin{lemma}
Let $n\geq 2$, $a_{2}^{(1)},\ldots ,a_{2}^{(n+1)}\in K_{2}\cup H_{2};$ $b_{-2}^{(1)}, \ldots ,b_{-2}^{(n)}\in K_{-2}\cup H_{-2}$
and $Card(\{b_{-2}^{(j)}, j=1,\ldots,n\})<n.$ Then
$$a_{2}^{(1)}b_{-2}^{(1)}\cdots a_{2}^{(n)}b_{-2}^{(n)}a_{2}^{(n+1)}\in\sum
a_{2}^{(i_{1})}b_{-2}^{(j_{1})}a_{2}^{(i_{2})}\cdots
b_{-2}^{(j_{r})}a_{2}^{(i_{r+1})}(K_{-2}K_{2}+H_{-2}H_{2}),\, r<n.$$
\end{lemma}

\begin{proof}

Suppose that $a_{2}^{(n+1)}\in K_{2}$. If $b_{-2}^{(n)}\in K_{-2}$, then we
are done. Let $b_{-2}^{(n)}\in H_{-2}$. Suppose that there exists $1\leq
i\leq n-1$, such that $b_{-2}^{(i)}\in K_{-2}$. Then
$$ a_{2}^{(1)}\cdots b_{-2}^{(i)}\cdots a_{2}^{(n)}b_{-2}^{(n)}a_{2}^{(n+1)}=a_{2}^{(1)}\cdots \{b_{-2}^{(i)}\cdots b_{-2}^{(n)}\}a_{2}^{(n+1)}\pm a_{2}^{(1)}\cdots b_{-2}^{(n)}\cdots b_{-2}^{(i)}a_{2}^{(n+1)}$$ as claimed. Therefore we can assume that
$b_{-2}^{(i)}\in H_{-2}$, $1\leq i\leq n$. Let $a_{2}^{(n)}\in H_{2}$. Then
$$b_{-2}^{(n-1)}a_{2}^{(n)}b_{-2}^{(n)}a_{2}^{(n+1)}=b_{-2}^{(n-1)}(a_{2}^{(n)}b_{-2}^{(n)}a_{2}^{(n+1)}-a_{2}^{(n+1)}b_{-2}^{(n)}a_{2}^{(n)})
+b_{-2}^{(n-1)}a_{2}^{(n+1)}b_{-2}^{(n)}a_{2}^{(n)}$$
which is an element of $H_{-2}H_{2}+H_{-2}K_{2}H_{-2}H_{2}.$ We will assume therefore that $a_{2}^{(n)}\in K_{2}$.
Suppose that there exists $2\leq i\leq n-1$, such that $a_{2}^{(i)}\in H_{2}.$ Then
$a_{2}^{(i)}\cdots a_{2}^{(n)}b_{-2}^{(n)}a_{2}^{(n+1)}=(a_{2}^{(i)}b_{-2}^{(i)}\cdots a_{2}^{(n)}+(a_{2}^{(i)}b_{-2}^{(i)}\cdots a_{2}^{(n)})^{\ast})b_{-2}^{(n)}a_{2}^{(n+1)}$ \newline $\pm a_{2}^{(n)}\cdots a_{2}^{(i)}b_{-2}^{(n)}a_{2}^{(n+1)}$.
Both summands on the right hand side fall into the case that has just been
considered above.
From now on we will assume that $a_{2}^{(2)}$, \ldots , $a_{2}^{(n+1)}\in
K_{2};$ $b_{-2}^{(1)}$, \ldots , $b_{-2}^{(n)}\in H_{-2}$.
Notice that $b_{-2}^{(i-1)}a_{2}^{(i)}b_{-2}^{(i)}=%
\{b_{-2}^{(i-1)}a_{2}^{(i)}b_{-2}^{(i)}%
\}-b_{-2}^{(i)}a_{2}^{(i)}b_{-2}^{(i-1)}$.
The element $b_{-2}=\{b_{-2}^{(i-1)}a_{2}^{(i)}b_{-2}^{(i)}\}$ lies in $%
K_{-2}$.
Hence the product $a_{2}^{(1)}\cdots a_{2}^{(i-1)}b_{-2}a_{2}^{(i+1)}\cdots
a_{2}^{(n+1)}$ is one of those considered before. We proved that the
elements $b_{-2}^{(i)}$, $1\leq i\leq n$, in $a_{2}^{(1)}b_{-2}^{(1)}\cdots
a_{2}^{(n+1)}$ are skew-symmetric modulo expressions of the desired type.
Now taking into account that \newline$Card(b_{-2}^{(i)}$, $1\leq i\leq n)<n$, we
arrive at the conclusion of the lemma. We started with the assumption that $%
a_{2}^{(n+1)}\in K_{2}$. The case of $a_{2}^{(n+1)}\in H_{2}$ is treated
similarly. This finishes the proof of the Lemma.
\end{proof}

\subsubsection*{Proof of Theorem 2.}

\begin{proof}

By Lemma 2 the associative pair $(R_{-2}$, $R_{2})$ is
finitely generated. Without loss of generality we will assume that $(R_{-2}, R_{2})$ is generated by elements $a_{2}^{(i)}$, $b_{-2}^{(j)}\in K\cup H,$ $1\leq i,j\leq n$.
Consider the set of products
$P=P_{-2}\cup P_{2},$ \newline $P_{2}=\{a_{2}^{(i_{1})}b_{-2}^{(j_{1})}\cdots a_{2}^{(i_{r})}|1\leq r\leq n+2\}$,
$P_{-2}=\{b_{-2}^{(j_{1})}a_{2}^{(i_{1})}\cdots b_{-2}^{(j_{r})}|1\leq r\leq n+2\}$. By Lemma 4 for an arbitrary product
$p\in P_{\pm 2}$ we have $p+p^{\ast }=\sum \alpha _{p,i}k_{p,i}^{2}$, where $\alpha _{p,i}\in F$, $k_{p,i}\in K_{\pm 1}$.
By Lemma 5 there exist finite sets $M_{-1}\subset K_{-1}$, $M_{1}\subset
K_{1}$ such that $R_{1}=M_{-1}R_{2}+R_{2}M_{-1}$, $%
R_{-1}=M_{1}R_{-2}+R_{-2}M_{1}$.

We claim that the Lie algebra $[K, K]$ is generated by the union of the sets: $M_{1},M_{-1}$, $\{[\{p\},\{q\}] | p, q\in P\},$
$\{[(p+p^{\ast })\circ k_{q,i},$ $k_{q,i}] | p,q\in P\},$ $\{M_{-1}P_{2}\},\{M_{1}P_{-2}\}$. By Lemma 4 and Lemma 6 it is
sufficient to prove that all elements from $K_{-1}$, $K_{1}$ can be
expressed by these elements.
By Lemma 5 $K_{1}$ is spanned by elements of
the type $\{k_{-1}a_{2}^{(i_{1})}b_{-2}^{(j_{1})}\cdots a_{2}^{(i_{r})}\}$, $%
k_{-1}\in M_{-1}$. We will use induction on $r$.

If $r\leq n+2$ then the assumption is clear. If $r>n+2$ then by Lemma 7 applied to
$a^{(i_{r-n+2})}b^{(j_{r-n+2})}\ldots a^{(i_{r-1})}b^{(j_{r-1})}a^{(i_r)}$ the element
$k_{-1}a_{2}^{(i_{1})}b_{-2}^{(j_{1})}\cdots a_{2}^{(i_{r})}$ is a linear
combination of elements of the type\newline  $k_{-1}a_{2}^{(\mu _{1})}b_{-2}^{(\nu
_{1})}\cdots a_{2}^{(\mu _{t})}k_{-2}^{\prime }k_{2}^{\prime }$ and $%
k_{-1}a_{2}^{(\mu _{1})}b_{-2}^{(\nu _{1})}\cdots a_{2}^{(\mu
_{t})}h_{-2}^{\prime }h_{2}^{\prime }$, where $t<r;$ \newline $k_{-2}^{\prime }\in
\{P_{-2}\}$, $k_{2}^{\prime }\in \{P_{2}\};$ $h_{-2}^{\prime }$, $%
h_{2}^{\prime }\in \{p+p^{\ast }$ $|$ $p\in P\}$. We have
\begin{alignat*}{5}
\{k_{-1}a_{2}^{(\mu _{1})}b_{-2}^{(\nu _{1})}\cdots a^{(\mu
_{t})}k_{-2}^{\prime }k_{2}^{\prime }\}&=[\{k_{-1}a_{2}^{(\mu _{1})}\cdots
a_{2}^{(\mu _{t})}\}, [k_{-2}^{\prime }, k_{2}^{\prime }]],\\
\{k_{-1}a_{2}^{(\mu _{1})}b_{-2}^{(\nu _{1})}\cdots a_{2}^{(\mu
_{t})}h_{-2}^{\prime }h_{2}^{\prime }\}
&=[\{k_{-1}a_{2}^{(\mu _{1})}\cdots
a_{2}^{(\mu _{t})}\}, [h_{-2}^{\prime }, h_{2}^{\prime }]]
\end{alignat*}
Now it remains to notice that if $h_{2}^{\prime }=p+p^{\ast }$, $p\in P$,
then $h_{2}^{\prime }=\sum \alpha _{p\text{, }i}k_{p\text{, }i}^{2}$ and

$[h_{-2}^{\prime }$, $k_{p\text{, }i}^{2}]=2[h_{-2}^{\prime }\circ k_{p\text{%
, }i}$, $k_{p\text{, }i}]\in \lbrack K_{-1}$, $K_{1}]$.

This finishes the proof of the theorem.
\end{proof}

\section{Simple algebras}

Let $R$ be a simple finitely generated $F-$ algebra with an involution $*:R\rightarrow R,$ char $F\neq 2,$ e is an idemotent such that $ee^{*}=e^{*}e=0,$ $K=\{a \in R \mid a^{*}=-a\}.$ \newline If $e+e^{*}$ is not an identity of $R$ then the Lie algebra $[K,K]$ is finitely generated by theorem 2. Suppose that $e+e^{*}=1.$ As above, let $R_{-2}=e R e^{*},$ \newline$ R_{0}=e R e+e^{*} R e^{*}, R_{2}=e^{*} R e.$ Then $R=R_{-2}+R_{0}+R_{2}$ is a $\mathbb{Z}$- grading of $R.$ Denote $K_{i}=K\bigcap R_{i},i=-2,0,2.$

 If $R$ has a nonzero center $Z$ and $dim_{Z}R<\infty$ then  $Z$ is a finitely generated $F$-algebra. Since $Z$ is a field it follows that $dim_{F}Z<\infty,$ hence $dim_{F}R<\infty$ and $dim_{F}[K,K]<\infty.$ From now on we will assume that the algebra $R$ is not finitely dimensional over its center.
\begin{lemma}
The algebra $R$ is generated by $K_{-2}+K_{2}.$
\end{lemma}
\begin{proof}
Since the algebra $R$ is not finite dimensional over its center it follows that $R$ does not satisfy a polynomial identity. By the result of S. Amitsur [1] $R$ does not satisfy a polynomial  identity with involution. Hence $[[[K,K],K],[[K,K],K]]\neq (0).$ I.Herstein [4] proved that if $A$ is a subalgebra of $R$ such that $[A,K]\subseteq A,A$ is not commutative and $dim_{Z}R >16,$ then $A=R.$ Applying this result to the associative subalgebra generated by $[[K,K],K]$ we see that $[[K,K],K]$ generates $R$.

Let us show that $[[K,K],K]\subseteq K_{-2}+[K_{-2},K_{2}]+K_{2}.$ The Lie algebra $K_{-2}+[K_{-2},K_{2}]+K_{2}$ is an ideal in the Lie algebra $K=K_{-2}+K_{0}+K_{2}$ and, hence, in the Lie algebra $[K,K].$ As shown by W. E. Baxter [2] the Lie algebra $[K,K]/[K,K]\bigcap Z$ is simple. Hence $[K,K]\subseteq K_{-2}+[K_{-2},K_{2}]+K_{2}+Z.$ Now, $[[K,K],K]\subseteq [K_{-2}+[K_{-2},K_{2}]+K_{2},K]\subseteq K_{-2}+[K_{-2},K_{2}]+K_{2}.$ This finished the proof of the lemma.
\end{proof}

\begin{lemma}
Let $A$ be a semi prime $F-$ algebra with in involution $*:A\rightarrow A,$ $\text{ char } F\neq 2,K=K(A,*)=\{a \in A \mid a^{*}=-a\}.$ Suppose that $k \in K$ and $k K k=(0).$ Then $k=0.$
\end{lemma}

\begin{proof}
Let $H=\{a \in A \mid a^{*}=a\}.$ Clearly, $A=K+H.$ If $k\neq 0$ then $k A k =k H k \neq (0).$ Choose an element $h \in H$ such that $k hk\neq 0.$ We have $(khk)K (khk)=(0).$ For an arbitrary element $h_{1} \in H$ we have  $k h k h_{1} k h k=k(h k h_{1}+h_{1}k h) k h k -k h_{1} k h k h k=0,$
since $h k h_{1}+h_{1}k h\in K$ and $h k h \in K.$ This implies $(k h k) A (k h k)=(0),$ which contradicts the semi primeness of $A.$ Lemma is proved.\\

Let $F<X>=F<x_{1},...,x_{m}>$ be the free associative algebra without $1.$ The mapping $x_{i}\rightarrow -x_{i},1\leq i \leq m,$ extends to the involution \newline$*:F<X>\rightarrow F<X>.$ For an arbitrary generator $x_{i},$ arbitrary elements $a_{1},a_{2},a_{3} \in K(F<X>,*)$ denote $f (x_{i},a_{1})=x_{i}a_{1}x_{i},f (x_{i},a_{1},a_{2})=f(x_{i}a_{1})a_{2} f(x_{i},a_{1}),$ $f(x_{i},a_{1},a_{2},a_{3})=f(x_{i},a_{1},a_{2})a_{3} f(x_{i},a_{1},a_{2}).$

Let $I$ be the ideal of the algebra $F<X>$ generated by all elements \newline $f(x_{i},a_{1},a_{2},a_{3}),1 \leq i \leq m, a_{1},a_{2},a_{3} \in K (F<X>),*).$\\

Recall that the Baer radical $B(A)$ of an associative algebra $A$ is the smallest ideal of $A$ such that the factor-algebra $A/B(A)$ is semi prime. The Baer radical is locally nilpotent: an arbitrary finite collection of elements from $B(A)$ generates a nilpotent subalgebra (see [3]).
\end{proof}

\begin{lemma}
The factor-algebra $F<X>/I$ is nilpotent.
\end{lemma}

\begin{proof}
Since $I^{*}=I,$ the involution  $*$ gives rise to an involution on $A=F<X>/I$. Let $B(A)$ be the Baer radical of $A$, $\overline{A}=A/B(A).$ The radical $B(A)$ is invariant with respect to any involution, hence $\overline{A}$ is an involutive semi prime algebra. By Lemma 9 the image of an arbitrary element $f(x_{i},a_{1},a_{2},a_{3}),a_{1},a_{2},a_{3} \in K(F<X>,*),$ is equal to zero in $\overline{A}.$ Again, subsequently applying lemma 9 three times we get $f(x_{i},a_{1},a_{2})=0$ in $\overline{A}$, $f(x_{i},a_{1})=0$ in $A$ and, finally, $x_{i}=0$ in $\overline{A}$,which means that $A=B(A).$ Since the algebra $A$ is finitely generated we conclude that $A$ is nilpotent. Lemma is proved.
\end{proof}

The degree of nilpotency of the algebra $A$ depends on the number of generators $m.$ Let $F<x_{1},...,x_{m}>^{d(m)}\subseteq I.$\\
\newline Now let us return to our finitely generated simple algebra $R.$ By Lemma 8 there exists elements $k_{1}^{+},...,k_{m}^{+} \in K_{2}, k_{1}^{-},...,k_{m}^{-} \in K_{-2}$ that generates $R.$

\begin{lemma}
The Jordan pair $(K_{-2},K_{2})$ is generated by elements $\{a\},$ where $a$ are products in $k_{i}^{\pm}, 1\leq i \leq m$ generated by products of odd length $< d(2m)$.
\end{lemma}

\begin{proof}
Choose a generator $k_{i}^{\sigma},\sigma =+\, or \, -,$ and three elements $c_{1},c_{2},c_{3} \in K_{-\sigma 2}.$ Denote $f_{1}=k_{i}^{\sigma}c_{1}k_{i}^{\sigma},f_{2}=f_{1}c_{2}f_{1},f_{3}=f_{2}c_{3}f_{2}.$ \newline Choose arbitrary elements $a_{1},...,a_{n} \in K_{-\sigma 2}; b_{1},...,b_{n} \in K_{\sigma 2}.$

\begin{enumerate}
   \item Let $u=a_{1}b_{1}..a_{n}b_{n}.$ Then $f_{1}u=k_{i}^{\sigma} c_{1}k_{i}^{\sigma}u=k_{i}^{\sigma}c_{1}\{k_{i}^{\sigma}u\}\pm k_{i}^{\sigma}(u^{*}c_{1})k_{i}^{\sigma}.$ Hence $\{f_{1}u\}=\{k_{i}^{\sigma},c_{1},\{k_{i}^{\sigma}u\}\}\pm \{ k_{i}^{\sigma},\{c_{1}u\},k_{i}^{\sigma}\}$ is a nontrivial Jordan expression.

 \item Let $ u=b_{1}a_{1}...b_{t}a_{t},v=a_{t+1}b_{t+1}...a_{n}b_{n}.$ Then $\{ u f_{2}v\}=\{\{u f_{1}\},c_{2},\{f_{1}u\}\}\pm \{f_{1}...\}$ is a nontrivial Jordan expression by (1).

     \item Let $u=a_{1}b_{1}..b_{t-1}a_{t},v=a_{t+1}b_{t+1}...a_{n-1}b_{n-1}a_{n}$, it remains to consider the product $u f_{3}v.$ We have $f_{3}=f_{2}c_{3}f_{2}=f_{1}c_{2}f_{1}c_{3}f_{1}c_{2}f_{1}=k_{i}^{\sigma}c_{1} k_{i}^{\sigma}c_{2} k_{i}^{\sigma}c_{1}k_{i}^{\sigma}c_{3}k_{i}^{\sigma}c_{1}k_{i}^{\sigma}c_{2}k_{i}^{\sigma}c_{1}k_{i}^{\sigma}=k_{i}^{\sigma}c'_{2} k_{i}^{\sigma}c_{3} k_{i}^{\sigma}c'_{2} k_{i}^{\sigma},$ where $c'_{2}=c_{1}k_{i}^{\sigma}c_{2}k_{i}^{\sigma}c_{1}.$\\

         Hence, $\{uk_{i}^{\sigma}c'_{2} k_{i}^{\sigma}c_{3} k_{i}^{\sigma}c'_{2} k_{i}^{\sigma}v\}=\{\{u k_{i}^{\sigma}c'_{2}\},k_{i}^{\sigma}c_{3} k_{i}^{\sigma},\{c'_{2} k_{i}^{\sigma}v\}\}\pm \{c'_{2}...\} $

 \end{enumerate}

The second summand can be treated in the same way as we did in (1).

Now we are ready to finish the proof of the lemma

Let $a=k_{i1}^{\sigma},k_{j1}^{-\sigma}...k_{is}^{\sigma},\, 2s-1\geq d(2m).$ Then by lemma 11 ${a}$ is a Jordan expression in elements ${b},$ where $b$ are products in $k_{i}^{\pm},1 \leq i \leq m,$ of odd length less than $2s-1.$ This proves the lemma.
\end{proof}
As we have already mentioned above $[K,K]+Z/Z=K_{-2}+[K_{-2},K_{2}]+K_{2}+Z/Z.$
In view of Lemma11 this implies that the algebra $[K,K]/[K,K] \bigcap Z$ is finitely generated. Theorem 3 is proved.

\section*{Acknowledgement}
This paper was funded by King Abdulaziz University, under grant No. (12-130-1434 HiCi). The authors, therefore, acknowledge technical and financial support of KAU.

\end{document}